\documentclass[11pt]{article}
\usepackage{amsmath,amsthm,amssymb,fullpage}

\begin{document}
\newcounter{thmctr}
\newtheorem{theorem}[thmctr]{Theorem}
\newtheorem{corollary}[thmctr]{Corollary}
\newtheorem{lemma}[thmctr]{Lemma}
\newtheorem{proposition}[thmctr]{Proposition}
\newtheorem{conjecture}[thmctr]{Conjecture}
\newtheorem{definition}[thmctr]{Definition}

\newcommand{\rt}{\mathbf{RT}}
\newcommand{\diam}{\text{diam}}
\newcommand{\caps}{\text{cap}}
\newcommand{\ex}{\text{ex}}
\newcommand{\tkop}{\mathrm{TK}}
\newcommand{\tkfop}{\mathcal{TK}}
\newcommand{\HH}{\mathcal{H}}
\newcommand{\diamt}[1]{d_\text{max}(#1)}
\newcommand{\tk}[2]{\tkop^{#2}(#1)}
\newcommand{\tkf}[2]{\tkfop^{#2}(#1)}

\newcommand{\jozsithanks}{This material is based upon work supported by NSF CAREER Grant DMS-0745185,
UIUC Campus Research Board Grants 11067, 09072 and 08086, and OTKA Grant K76099.}
\newcommand{\jozsiuni}{University of Illinois at Urbana-Champaign \\ jobal@math.uiuc.edu}
\newcommand{\johnuni}{University of Illinois at Chicago \\ lenz@math.uic.edu}

\title{Some Exact  Ramsey-Tur\'an Numbers}
\author{J\'ozsef Balogh \thanks{\jozsithanks} \\ \jozsiuni \and
        John Lenz \\ \johnuni}

\maketitle

\newcommand{\appendixsummary}{The appendix contains a sketch of the proof of
Proposition~\ref{appendixprop}.}

\begin{abstract}
Let $r$ be an integer, $f(n)$ a function, and $H$ a graph.
Introduced by  Erd\H{o}s, Hajnal, S\'{o}s, and Szemer\'edi~\cite{rt-erdos83}, the $r$-Ramsey-Tur\'{a}n number of $H$,
$\rt_r(n, H, f(n))$, is defined to be 
the maximum number of edges in an $n$-vertex, $H$-free graph $G$ with
 $\alpha_r(G)\leq f(n)$ where $\alpha_r(G)$
 denotes the $K_r$-independence number of $G$. 
 
In this note, using isoperimetric properties of the high dimensional unit sphere,  we construct 
graphs providing lower bounds for $\rt_r(n,K_{r+s},o(n))$ for every $2\le s\le r$. 
These constructions are sharp for an infinite family of pairs of $r$ and $s$.
The only previous sharp construction (for such values of $r$ and $s$) was by Bollob\'as and Erd\H os \cite{beg-bollobas76}
for $r = s = 2$.
\end{abstract}

\section{Introduction} \label{secintro}

Let $G$ be a graph and define the \emph{$K_r$-independence number} of $G$ as  $$\alpha_r(G):= 
\max \left\{ \left| S \right| : S \subseteq V(G), G[S] \text{ is } K_r \text{-free} \right\}.$$  
Define $\rt_r(n,H,f(n))$ to be the maximum number of edges in an $H$-free graph $G$ on $n$ vertices with $\alpha_r(G) \leq f(n)$ and let
\begin{align} \label{rtdef}  
 \theta_r({H})=  \lim_{\epsilon \rightarrow 0} \lim_{n \rightarrow \infty} \ \frac{1}{n^2} \ \rt_r(n,H,\epsilon n).
\end{align}
We write $\rt_r(n,H,o(n))= \theta_r({H}) n^2 + o(n^2)$.
For $r = 2$, it is easy to show that the limit in \eqref{rtdef} exists; for $r \geq 3$,  its existence was proved when $H$ is a  complete graph in \cite{rt-erdos94}.
The \emph{$r$-Ramsey-Tur\'{a}n  number}  of $H$ is $\theta_r({H})$.

Tur\'an's Theorem \cite{tur-turan41} states that the maximum number of edges in a $K_r$-free graph on $n$ vertices is achieved by the complete $(r-1)$-partite graph. 
This extremal graph has independent sets with  linear size, which motivated
Erd\H os and S\'os~\cite{rt-erdos70} to ask about the maximum number of edges in a $K_r$-free
graph on $n$ vertices with sublinear independence number.
They solved this problem when $r$ is an odd integer.  The case when $r$ is even has a more
 interesting history.
Szemer\'edi \cite{rt-szemeredi72} used an early version of the Szemer\'edi Regularity Lemma to upper
bound $\theta_2(K_4)$ by $\frac{1}{8}$. This turned out to be sharp as 
four years later  Bollob\'as and Erd\H os~\cite{beg-bollobas76}
constructed $K_4$-free graphs with $n^2/8-o(n^2)$ edges and sublinear independence number.
Erd\H{o}s, S\'{o}s, Hajnal, and Szemer\'edi~\cite{rt-erdos83} extended these results to determine
$\theta_2(K_{2r})$ for all $r \geq 2$.

Another Ramsey-Tur\'{a}n result is 
an important and widely applicable theorem of Ajtai,  Koml\'os, and Szemer\'edi~\cite{ram-ajtai80}.
They lower bounded the independence number of triangle-free, $n$-vertex graphs with
$m$ edges.  Their result can be phrased as
\begin{align} \label{akseq}
\rt_2\left(n,K_3,\frac{cn^2}{m} \log\left( \frac{m}{n} \right)\right) < m
\end{align}
for some constant $c$.  This result imples a sharp upper bound of $cn^2/\log n$ on
the Ramsey number $R(3,n)$.
Other applications of \eqref{akseq} include Ajtai, Koml\'os, and  
Szemer\'edi's~\cite{ram-ajtai-sidon} impovements on Erd\H os and Tur\'an's~\cite{ram-erdos41}
result on the existence of dense infinite Sidon sets.
Recently, Fox~\cite{ram-fox10} used \eqref{akseq} to find 
large clique-minors in graphs with independence number two.
Hypergraph variants of \eqref{akseq} by Ajtai, Koml\'os,  Pintz,  Spencer  
and Szemer\'edi~\cite{ram-ajtai-hyper} have been applied by 
Koml\'os,  Pintz and  Szemer\'edi~\cite{ram-komlos82} in discrete computational geometry to
provide a counterexample for Heilbronn's Conjecture.  See \cite{rt-simonovits01} for 
a more detailed history of Ramsey-Tur\'{a}n numbers.

This paper focuses on the problem of determining $\theta_r(K_t)$ for $r \geq 3$, suggested by Erd\H{o}s, Hajnal, S\'{o}s, and Szemer\'{e}di~\cite[p. 80]{rt-erdos83} 
(see also \cite[Problem 17]{rt-simonovits01}).
Erd\H{o}s, Hajnal, Simonovits, S\'{o}s, and Szemer\'{e}di~\cite{rt-erdos94} proved that $\theta_r(K_t) \leq \frac{1}{2} \left( 1 - \frac{r}{t-1} \right)$ and this 
is best possible for all $t \equiv 1 \pmod r$.  This left open the question when $t \not\equiv 1 \pmod r$, where they made partial progress for $s\le \min\{5,r\}$.

\begin{theorem}\label{erdosupper} 
For $2 \leq s \leq \min\left\{ 5,r \right\}$, $\rt_r(n,K_{r+s},o(n))\leq \frac{s-1}{4r} n^2 + o(n^2)$.
\end{theorem}

Our main result is to construct for
every $2 \leq s \leq r$ an infinite graph family providing near-optimal 
lower bounds for  $\rt_r(n,K_{r+s},o(n))$.
In particular, we show that Theorem~\ref{erdosupper}
is sharp when $4r/(s-1)$ is a power of $2$. Earlier the only sharp construction was by Bollob\'as and Erd\H os \cite{beg-bollobas76}
for $r = s = 2$.

\begin{theorem}\label{mainconst}
Let $2\le s\le r$. Let $\ell$ be the largest positive integer such that 
$\lceil r\cdot 2^{-\ell}\rceil  <s$. Then 
$$\rt_r(n,K_{r+s},o(n))\ge  2^{-\ell-2} n^2.$$
\end{theorem}
For example, it yields that $\theta_{4}(K_{6})=1/16$ and $\theta_{4}(K_{7})=1/8$. 
We suspect that Theorem~\ref{mainconst} should be best possible for all $s$ when $4r/(s-1)$ is a power of $2$; towards this direction we have only the following partial result
extending Theorem~\ref{erdosupper}.
\begin{proposition}\label{appendixprop}
$\theta_{10}(K_{16})=\theta_{12}(K_{19})=\frac{1}{8}$.
\end{proposition}

The authors~\cite{rt-balogh11} recently proved  $\theta_r(K_{r+2}) > 0$  for every $r \geq 2$.  This resolved one of the main open questions
from \cite{rt-erdos94}. In~\cite{rt-balogh11} hypergraphs were constructed to estimate Ramsey-Tur\'an numbers of some  hypergraphs. Taking the shadow graphs of the constructed
hypergraphs implied the results for graphs.
Our proof builds on the techniques developed in~\cite{beg-bollobas76} and~\cite{rt-balogh11} combined with several new ideas.

The remainder of this paper is organized as follows:  in Section~\ref{secconstruction} we describe the construction for the graphs used to prove
Theorem~\ref{mainconst}, in Section~\ref{secproof} we prove Theorem~\ref{mainconst}, and in Section~\ref{secend} we list several open problems.
\appendixsummary

\section{Construction} \label{secconstruction}

The construction for Theorem~\ref{mainconst} builds on the Bollob\'{a}s-Erd\H{o}s
Graph~\cite{beg-bollobas76}.
The reader is encouraged to read Section~4 and the first few paragraphs of Section~5 from \cite{rt-balogh11}, which provide overviews some of the previous constructions.

First we briefly sketch a few properties of the unit sphere.  For more details, see Section~3 of \cite{rt-balogh11}.
Let $\mu$ be the Lebesgue measure on the $k$-dimensional unit sphere $\mathbb{S}^k \subseteq \mathbb{R}^{k+1}$ normalized so that $\mu(\mathbb{S}^k) = 1$.
Given any $\alpha, \beta > 0$, it is possible to select $\epsilon > 0$ small enough and then $k$ sufficiently large  so that Properties~(P1) and (P2)
are satisfied.
\begin{itemize}
\item[(P1)] Let $C$ be a spherical cap in $\mathbb{S}^k$ with height $h$, where $2h = \left( \sqrt{2} - \epsilon/\sqrt{k} \right)^2$
 (this means all points of the spherical cap are within distance $\sqrt{2} - \epsilon/\sqrt{k}$ of the center).
  Then $\mu(C) \geq \frac{1}{2} - \alpha$.
\item[(P2)] Let $C$ be a spherical cap with diameter $2 - \epsilon/(2r^2\sqrt{k})$.  Then $\mu(C) \leq \beta$.
\end{itemize}

To prove Theorem~\ref{mainconst}, it suffices to prove that for all integers $n,r \geq 2$, every $2 \leq s \leq r$, and every $\alpha, \beta > 0$, there exists an $N$-vertex
graph $G = G(n,r,s,\alpha,\beta)$ such that $G$ is $K_{r+s}$-free, $N \geq n$, and
\begin{align*}
\left| E(G) \right| \geq \left( 2^{-\ell - 2} - \alpha \right) N^2 \quad \text{and} \quad \alpha_r(G) < \beta N,
\end{align*}
where $\ell$ is the largest positive integer such that $\left\lceil r \cdot 2^{-\ell} \right\rceil < s$.

Assume  $n$, $r$, $s$, $\alpha$, $\beta$, and $\ell$ are given as above, we shall show how to construct $G = G(n,r,s,\alpha,\beta)$.  For the given $\alpha$ and $\beta$, 
there exists $\epsilon > 0$ and $k \geq 3$ such that properties (P1) and (P2) hold.  Define $\theta = \epsilon/\sqrt{k}$ and  $z = 2n$.
Partition the $k$-dimensional unit sphere $\mathbb{S}^k$ into $z$ domains having equal measures and diameter at most $\theta/4$.  Choose a point from each set
and let $P$ be the set of these points.  Let $\phi : P \rightarrow \mathcal{P}(\mathbb{S}^k)$ map points of $P$ to the corresponding domains of the sphere.
Before defining $G$, we construct some auxiliary bipartite graphs $B_1, \ldots, B_{\ell}$ and hypergraphs $\mathcal{H}$ and $\mathcal{H}'$.

The vertex set of the auxiliary bipartite graphs $B_1, \dots, B_{\ell}$ is $[r]$, 
and the edges are built from the $\ell$-dimensional hypercube $Q_{\ell}$ as follows.
Blow up $Q_{\ell}$ into $Q'_{\ell}$ so that each vertex is blown up into
an independent set of size $s-1$.  Discard vertices of $Q'_{\ell}$ so that $Q'_{\ell}$ has exactly
$r$ vertices, discarding at most one vertex from each  blow up class. (Note that $\ell$ was chosen so that $Q_{\ell}$ is the smallest hypercube with at least $r/(s-1)$ vertices.)
Consider the vertices of $Q_{\ell}$ as labeled by binary words of length $\ell$.
If the $(2i+1)$-st discarded vertex is from the class labelled by $(a_1,\ldots,a_\ell)$,
then the $(2i+2)$-nd vertex should be removed from the class labelled $(1-a_1,\ldots,1-a_\ell)$.
Denote by
$A_{i,0}$ the subset of vertices of $Q'_{\ell}$ which come from a blowup of a vertex with its $i$-th coordinate zero.  Similarly define $A_{i,1}$.
The bipartite graph $B_i$ is the complete bipartite graph with parts $A_{i,0}$ and $A_{i,1}$.

Now we define an $r$-uniform hypergraph $\mathcal{H}$ with vertex set $P^{\ell}$, the family of ordered $\ell$-tuples of elements of $P$.
We let $E \subseteq P^{\ell}$ be a hyperedge  of $\mathcal{H}$  if $|E| = r$ and there exists some ordering 
$\bar x^1,\ldots,\bar x^r$ of the elements of $E$ such that
 for every $1\le i<j\le r$ and $1\le a\le \ell$ if $ij \in E(B_a)$
 then $d(x^i_a, x^j_a) > 2 - \theta$.  
In other words, we form a hyperedge if the edges of $B_a$ correspond to almost antipodal points on the sphere in the $a$-th vertex coordinate.

From $\mathcal{H}$, define a hypergraph $\mathcal{H}'$ by applying the following theorem to
$\mathcal{H}$ with  $\gamma = \beta$ and $k = r^3$.
\begin{theorem} [Theorem 16 in \cite{rt-balogh11}]\label{randblowup}
Let $\mathcal{H}$ be an $r$-uniform hypergraph on $n$ vertices.  Let $0 < \gamma < 1$ and let $k$ be a positive integer.
Then there exists a $t = t(\mathcal{H},k,\gamma,r)$ and an $r$-uniform hypergraph $\mathcal{G}$ with vertex set 
$V(\mathcal{H}) \times [t]$ with the following properties.
\begin{itemize}
\item[(i)] For all $\left\{ a_1, \ldots, a_r \right\} \in \mathcal{H}$ and all sets
 $U_i \subseteq \left\{ a_i \right\} \times [t]$ with $\left| U_i \right| \geq \gamma t$ for each $1 \leq i \leq r$,
there exists at least one hyperedge of $\mathcal{G}$ with one vertex in each $U_i$.
\item[(ii)] $\mathcal{G}$ does not contain as a subhypergraph any
 $v$-vertex hypergraph $\mathcal{F}$ with $m$ edges where $v \leq k$ and $v + (1+\gamma - r)(m-1) < r$.
\end{itemize}
\end{theorem}
The proof of this theorem is a straightforward random argument: blow up the hypergraph $\mathcal{H}$
and randomly delete edges similar to the proof of the existence of a graph with large girth and small
independence number.  See \cite{rt-balogh11} for more details.

We are finally ready to define $G$.  Let $U$ and $V$ be two distinct copies of $V(\mathcal{H}')$ and let the vertex set of $G$ be $U \dot{\cup} V$.
We  place a copy of the shadow graph of $\mathcal{H}'$ on  both $G[U]$ and $G[V]$.  (The \emph{shadow graph} of a hypergraph has the same vertex set and $xy$ forms
an edge of the shadow graph if $x$ and $y$ are contained together in some hyperedge.)  Lastly, 
for  $\bar u=\left<u_1,\ldots,u_{\ell}\right>\in U$ and $\bar v=\left<v_1,\ldots,v_{\ell}\right>\in V$ let
$\bar u\bar v$ be an edge if $d(u_i,v_i)<\sqrt 2-\theta$ for all $1 \leq i \leq \ell$.

This differs from the constructions in \cite{rt-balogh11} in two important places.
In \cite{rt-balogh11}, the cross-edges are defined when $d(u_i,v_j) < \sqrt{2} - \theta$ for all $1 \leq i,j \leq \ell$.  By weakening this to only require
$d(u_i,v_i) < \sqrt{2} - \theta$, the density of cross-edges is much larger. The cost is that
here we need to work harder to show these new edges do not create copies of $K_{r+s}$.
Secondly, where we used the auxiliary bipartite graphs $B_i$'s in the construction,
\cite{rt-balogh11} used trees.   The number of auxiliary graphs is $\ell$, the number of coordinates in our vertices.  The larger $\ell$ gets, the smaller the number of edges
since each additional coordinate imposes more distance requirements on points.  By switching from trees to bipartite graphs, we are able to use fewer
coordinates.  This makes $G[U]$ and $G[V]$ sparser, which forces a more complicated proof that $G$ has small independence number.

\section{Verifying properties of $G$} \label{secproof}

To complete the proof of Theorem~\ref{mainconst}, we need to prove three properties of $G$:
$G$ has at least $\left( 2^{-\ell-2} - \alpha \right) N^2$ edges, $G$ is $K_{r+s}$-free, and the $K_r$-independence number
of $G$ is smaller than $\beta N$.

\subsection{The number of edges of $G$.}

First we compute the number of vertices of $G$.  The hypergraph $\mathcal{H}$ has $z^\ell$ vertices and each vertex in $\mathcal{H}$ is blown up into a set of
size $t$ so $\mathcal{H}'$ has $tz^{\ell}$ vertices.  Thus $G$ has $2tz^{\ell}$ vertices.  To estimate the number of  edges of $G$,  we fix some vertex $x' \in U$;
we will compute a lower bound on its degree in $V$. There exists a vertex $x$ in $\mathcal{H}$
such that $x'$ is contained in the blowup of $x$.  For $y' \in V$ to be adjacent to $x'$, we must have
$d(x_i, y_i) \leq \sqrt{2} - \theta$ for all $i$.  By Property (P1), there are at least $\left(\frac{1}{2} - \alpha\right)\left| P \right|$ points $y_i$ that
are within distance $\sqrt{2} - \theta$ of $x_i$.  Thus there are at least $2^{-\ell}\left| P \right| - C\alpha \left| P \right|$ choices for $y$ where $C$ is some constant depending only on $\ell$.
Since each $y$ is blown up into a set of size $t$, the degree of $x'$ is at least $2^{-\ell}tz^{\ell} - C\alpha t z^{\ell}$.  Thus
\begin{align*}
\left| E(G) \right| \geq \frac{\left| V(G) \right|}{2} \left( 2^{-\ell} t z^{\ell} - C\alpha t z^{\ell} \right).
\end{align*}
Since $t z^{\ell} = \frac{\left| V(G) \right|}{2}$,
\begin{align*}
\left| E(G) \right| \geq 2^{-\ell-2} \left| V(G) \right|^2 - C \alpha \left| V(G) \right|^2/2= 2^{-\ell-2} \left| V(G) \right|^2 \left(1-C\alpha/2\right).
\end{align*}
Since $C$ depends only on $\ell$ and $\alpha>0$ can be chosen arbitrarily small, this gives the required bound.

\subsection{$G$ is $K_{r+s}$-free}

First we need a couple of short lemmas.

\begin{lemma} \label{Bindep}
$\alpha(\cup B_i) < s$.
\end{lemma}

\begin{proof}
Fix any  two vertices $x',y'\in V(Q'_{\ell})$ and let $x$ and $y$ be the vertices of $Q_{\ell}$ such that
$x'$ and $y'$ are contained in the blowups of $x$ and $y$ respectively.  If $x \neq y$, then their binary labels differ in at least one position
so there will be some $B_i$ where $x'$ and $y'$ appear in different classes of the bipartition of $B_i$.  Thus the independent sets in $\cup B_i$ are subset of the  blowup of some 
vertex in $Q_{\ell}$.  Using that each vertex in $Q_{\ell}$ is blown up into a set of size
at most $s-1$, the proof is complete.
\end{proof}

\begin{lemma} \label{compGsingleedge}
Let $K_w$ be a complete $w$-vertex subgraph of $G[U]$.  Then there exists a hyperedge $E$ in $\mathcal{H}'$
such that $V(K_w) \subseteq E$.
\end{lemma}

\begin{proof}
Let $K_w \subseteq G[U]$ and $V(K_w) = \left\{ x_1, \ldots, x_w \right\}$.  Since $K_w$ is complete, for every $i,j$ there exists some hyperedge $E_{i,j}$ of $\mathcal{H}'$ such that
$E_{i,j}$ contains both $x_i$ and $x_j$.  If the $E_{i,j}$'s are not all the same  hyperedge, then  \textit{(ii)} of Theorem~\ref{randblowup} is violated.
\end{proof}

\begin{lemma} \label{HisKr1free}
$G[U]$ (and similarly $G[V]$) is $K_{r+1}$-free.
\end{lemma}

\begin{proof}
This is an  immediate corollary of Lemma~\ref{compGsingleedge}.  Hyperedges in $\mathcal{H}$ have size at most $r$, so $G[U]$ does not contain
any $K_{r+1}$.
\end{proof}

We now need the following property of the unit sphere observed by Bollob\'{a}s and Erd\H{o}s~\cite{beg-bollobas89}.

\begin{theorem} [Bollob\'as-Erd\H os Rombus Theorem]\label{BErombus}
For any $0 < \gamma < \frac{1}{4}$, it is impossible to have $p_1, p_2, q_1, q_2 \in \mathbb{S}^k$ such that $d(p_1, p_2) \geq 2 - \gamma$,
$d(q_1, q_2) \geq 2 - \gamma$, and $d(p_i,q_j) \leq \sqrt{2} - \gamma$ for all $1 \leq i,j \leq 2$.
\end{theorem}

Recall that from the hypergraph $\mathcal{H}$ we formed the hypergraph $\mathcal{H}'$ by blowing
up each vertex in $\mathcal{H}$ into a strong independent set in $\mathcal{H}'$.  Also recall that the vertices in $G[U]$ are vertices
of $\mathcal{H}'$, so vertices in $G[U]$ correspond to blowups of vertices in $\mathcal{H}$.  We define a function $\Xi$ between $V(G)$ and
$V(\mathcal{H})$:  for $x \in V(G)$, let $\Xi(x)$ be the vertex of $V(\mathcal{H})$ such that $x$ is contained in the blowup of $\Xi(x)$.

\begin{lemma} \label{GisKrsfree}
$G$ is $K_{r+s}$-free.
\end{lemma}

\begin{proof}
Towards a contradiction, assume that $K=K_{r+s}$ is a subgraph of $G$ and let $K_u = K[V(K) \cap U]$ and $K_v = K[V(K) \cap V]$.  Since $U$ and $V$ are symmetric in the definition of $G$,
we may assume without loss of generality that $|K_u| \geq |K_v|$.  By Lemma~\ref{HisKr1free} and since $s \geq 2$, 
\begin{align} \label{eqsizeofKu}
\left\lceil r/2 \right\rceil + 1 \leq \left\lceil\frac{r+s}{2}\right\rceil \leq \left| V(K_u) \right| \leq r.
\end{align}
This implies that
\begin{align} \label{eqsizeofKv}
\left| V(K_v) \right| = r +s - \left| V(K_u) \right| \geq s.
\end{align}

By Lemma~\ref{compGsingleedge} and \eqref{eqsizeofKv}, there exist $x_1, \ldots, x_s \in V(K_v)$ and a hyperedge $E$ in $\mathcal{H}'$ such that $x_1, \ldots, x_s \in E$.
Since $x_1, \ldots, x_s$ are all in $E$ and edges of $\mathcal{H}$ were built from the auxiliary bipartite graphs $B_1, \ldots, B_{\ell}$, we can think of $\Xi(x_1), \ldots, \Xi(x_s)$ as
vertices in $\cup B_i$.  By Lemma~\ref{Bindep}, there exists some $B_i$ and two vertices, say $\Xi(x_1)$ and $\Xi(x_2)$, such that the $i$th coordinate of $\Xi(x_1)$
and the $i$th coordinate of $\Xi(x_2)$ are almost antipodal.  Fix this $i$ for the remainder of this proof.

By  \eqref{eqsizeofKu} there exist
at least ${\left\lceil r/2 \right\rceil + 1}$ vertices  in $V(K_u)$, say
$y_1, \ldots, y_{\left\lceil r/2 \right\rceil + 1}$. 
By Lemma~\ref{compGsingleedge} there is 
a hyperedge $F$ in $\mathcal{H}'$ containing them.  Similarly to the previous paragraph, we can think of $\Xi(y_1), \ldots, \Xi(y_{\left\lceil r/2 \right\rceil + 1})$
as vertices in $B_i$ (recall that $i$ has already been chosen.)
The parts of $B_i$ have size at most $\left\lceil r/2 \right\rceil$ so there exist two vertices, say $\Xi(y_1)$ and $\Xi(y_2)$, such that the $i$-th coordinates are
almost antipodal.

Consider the $i$-th coordinates of 
$\Xi(x_1)$, $\Xi(x_2)$, $\Xi(y_1)$, and $\Xi(y_2)$.  
The cross-distances between the $x$'s and $y$'s are all at most
$\sqrt{2} - \theta$, since $x_1$, $x_2$, $y_1$, and $y_2$ all came from the clique $K$.
Hence we have four points violating Theorem~\ref{BErombus}. 
\end{proof}

\subsection{The $K_r$-independence number of $G$}

First we need an elementary statement about distances of points on a sphere.

\begin{lemma}\label{trig}
Let $k \geq 2$ and $1 \leq h \leq \left\lfloor k/2 \right\rfloor$ be any positive integers
and fix a positive $a<1/(16h^4)$.  Let $x_1,\ldots,x_k\in \mathbb{S}^k$ such that for every $i$ we have $d(x_i,x_{i+1})\geq 2-a$.
Then $d(x_1,x_{2h})> 2-4h^2 a$.
\end{lemma}

\begin{proof}
For $u\in \mathbb{S}^k$ denote by $u'\in \mathbb{S}^k$ the antipodal point to $u$. 
Note  that for every $u,v$ trivially  $d(u,u')=2$ and $d(u,v)=d(u',v')$.  First, we bound $d(x'_i, x_{i+1})$ for every $i$.
The points $x_i$, $x'_i$, and $x_{i+1}$ form a right triangle since $x_i$ and $x'_i$ are antipodal (the right angle is at the point $x_{i+1}$).
Thus
\begin{align*}
d^2(x'_i,x_{i+1}) = d^2(x_i,x'_i) - d^2(x_i,x_{i+1}) \leq 4 - (2-a)^2 \leq 4a - a^2 \leq 4a.
\end{align*}
Thus $d(x'_i,x_{i+1}) \leq 2\sqrt{a}$ for all $i$.  Using the triangle inequality, we obtain
\begin{align*}
d(x'_1,x_{2h}) \leq d(x'_1,x_2) + d(x_2,x'_3) + \cdots + d(x'_{2h-1},x_{2h}) \leq 2(2h-1)\sqrt{a}.
\end{align*}
The points $x_1$, $x'_1$, and $x_{2h}$ form a right triangle since $x_1$ and $x'_1$ are antipodal.  Thus
\begin{align*}
d^2(x_1,x_{2h}) = d^2(x_1,x'_1) - d^2(x'_1,x_{2h}) \geq 4 - 4(2h-1)^2a = 4 - 16h^2a + 16ha - 4a \geq 4 - 16h^2a + a.
\end{align*}
Since $a < 1/(16h^4)$ implies  $16h^4a^2 < a$ we have
\begin{align*}
d^2(x_1,x_{2h}) \geq 4 - 16h^2a + a > 4 - 16h^2a + 16h^4a^2 = (2-4h^2a)^2.
\end{align*}
\end{proof}

We now need one lemma from \cite{rt-balogh11}.  There is a subtle point here: in \cite{rt-balogh11} the statement
of the lemma uses ``$d(p_i,p_j) \geq 2 - \theta$''.  But the variable $\theta$ used in this paper and the $\theta$ used in \cite{rt-balogh11}
are slightly different constants.  The $\theta$ used in the statement of \cite[Lemma 13]{rt-balogh11} comes from the
statement of \cite[Property (P3)]{rt-balogh11} which matches our Property (P2).  So the $\theta$ in \cite[Lemma 13]{rt-balogh11} is
replaced with the constant from our Property (P2) when we cite that lemma below.

\begin{lemma}(Lemma 13 in \cite{rt-balogh11}) \label{findtree}
If $A_1, \ldots, A_r \subseteq P$ with $\left|A_i\right| \geq 2^r\beta z$ and $T$ is
a tree on vertex set $[r]$, then
there exist $p_1 \in A_1, \ldots, p_r \in A_r$ such that if $ij \in E(T)$ then $d(p_i, p_j) \geq 2 - \epsilon /\left(2r^2 \sqrt{k}\right)$.
\end{lemma}

One of the key improvements in this paper compared to \cite{rt-balogh11} is improving the above lemma
by replacing trees with complete bipartite graphs.

\begin{lemma} \label{findbipartite}
If $A_1, \ldots, A_r \subseteq P$ with $\left|A_i\right| \geq 2^r\beta z$ and $B$ is
a complete bipartite graph on vertex set $[r]$, then
there exist $p_1 \in A_1, \ldots, p_r \in A_r$ such that if $ij \in E(B)$ then $d(p_i, p_j) \geq 2 - \theta$.
\end{lemma}

\begin{proof}
Let $T$ be a path on vertex set $[r]$.  Apply Lemma~\ref{findtree} to find $p_1 \in A_1, \ldots, p_r \in A_r$ such that
if $ij \in E(T)$ then $d(p_i,p_j) \geq 2 - \theta$.  Since $T$ is a path, this implies that $d(p_i,p_{i+1}) \geq 2 - \epsilon/\left( 2r^2\sqrt{k} \right) = 2-\theta/r^2$ for
all $i$.  We can then apply Lemma~\ref{trig} to show that $d(p_{2i+1},p_{2j}) > 2 - \theta$ for all $i$ and $j$
(set $x_1 = p_{2i+1}$ and $x_{2h} = p_{2j}$.)
\end{proof}

\begin{lemma} \label{Hsmallind}
$\alpha(\mathcal{H}) \leq r^{\ell} 2^{\ell + r} \beta z^{\ell}$.
\end{lemma}

\begin{proof}[Proof sketch]
The proof is identical to the proof of \cite[Lemma~14]{rt-balogh11}, except where \cite[Lemma 14]{rt-balogh11} uses Lemma~\ref{findtree} on trees, we instead use Lemma~\ref{findbipartite}.
\end{proof}

\begin{lemma} \label{constrsmallind}
$\alpha_r(G) \leq r^{\ell}2^{\ell+r+2}\beta z^{\ell} t$.
\end{lemma}

\begin{proof}[Proof sketch]
The proof is identical to the proof of \cite[Theorem 9 \textit{(iv)}]{rt-balogh11}.  The statement of Lemma~\ref{Hsmallind} is identical to the lemma
used by the proof of \cite[Theorem 9 \textit{(iv)}]{rt-balogh11}.
\end{proof}

\section{Concluding Remarks} \label{secend}

We conjecture that our construction is best possible when $4r/(s-1)$ is a power of $2$; we know this only when 
$s\leq 5$ and for some additional cases (see Proposition~\ref{appendixprop}). 
Probably, some mixture of a more involved application of the Szemer\'edi Regularity Lemma  and some proof techniques from weighted Tur\'an theory could help to prove our conjecture.

It seems very hard to decide if our constructions are best possible when  $4r/(s-1)$ is not a $2$-power. The simplest open cases are
$$\frac{1}{16} \leq \theta_3(K_5)\le \frac{1}{12}, \quad \quad \frac{1}{8} \leq \theta_3(K_6)\le \frac{1}{6}, \quad \quad  \frac{1}{8} \leq \theta_4(K_8)\le \frac{3}{16}.$$
The upper bounds are from Theorem~\ref{erdosupper} and the lower bounds from Theorem~\ref{mainconst}. 

Theorems~\ref{erdosupper} and~\ref{mainconst} focus on $\theta_r(K_t)$ for $t \leq 2r$.  What happens when $t > 2r$?
The construction from Theorem~\ref{mainconst} can be easily extended to cliques larger than $K_{2r}$ as follows.
For a lower bound on $\theta_r(K_{qr+s})$ with $2 \leq s \leq r$, let $G$ be the construction from Theorem~\ref{mainconst}
and join it to  a complete $(q-1)$-partite graph with almost equal part sizes.  Into each part insert a $K_{r+1}$-free graph with small $K_r$-independence number (such a graph
exists by the Erd\H{o}s-Rogers Theorem~\cite{rt-erdos62}.) Erd\H{o}s, Hajnal, Simonovits, S\'{o}s, and Szemer\'{e}di~\cite{rt-erdos94} conjectured that this type of construction
provides the correct answer (see also \cite[Conjecture 18]{rt-simonovits01}.)  Can Theorem~\ref{erdosupper} be extended
to $t > 2r$ and if so, does it match our lower bound?

In the area of the Ramsey-Tur\'{a}n theory, one of the major open problems is to prove a generalization of the Erd\H{o}s-Stone Theorem \cite{rrl-erdos46}
by proving that $\rt(n,H,o(n)) = \rt(n,K_s,o(n))$ where $s=s(H)$ is equal to some parameter depending only on $H$.  
Erd\H{o}s, Hajnal, S\'{o}s, and Szemer\'{e}di~\cite{rt-erdos83} proved an upper bound using a parameter closely related to the arboricity.  That is, 
one can take $s$ to be the minimum $s$ such that $V(H)$ can be partitioned into $\left\lceil s/2 \right\rceil$ sets $V_1, \ldots, V_{\left\lceil s/2 \right\rceil}$
such that $V_1, \ldots, V_{\left\lfloor s/2 \right\rfloor}$ span forests and if $s$ is odd $V_{\left\lceil s/2 \right\rceil}$ spans an independent set.
This is known to be sharp for odd $s$.
In several papers, Erd\H{o}s mentioned the problem of solving the simplest open case when $H = K_{2,2,2}$, where $s(K_{2,2,2})=4$, i.e., one would like to have
a lower bound of $\rt(n,K_{2,2,2},o(n)) \le \rt(n,K_4,o(n)) = \frac{1}{8} n^2 + o(n^2)$.
Even the question of determining if $\rt(n,K_{2,2,2},o(n)) = \Omega(n^2)$ is still open (see~\cite[Problem 4]{rt-simonovits01},~\cite[p. 72]{rt-erdos83},
\cite[Problem 1.3]{rt-sudakov03} among others).

\medskip

The authors would like to thank the referee for useful feedback and a careful reading of the
manuscript.

\bibliographystyle{abbrv}
\bibliography{refs}

\begin{thebibliography}{10}

\bibitem{ram-ajtai-hyper}
M.~Ajtai, J.~Koml{\'o}s, J.~Pintz, J.~Spencer, and E.~Szemer{\'e}di.
\newblock Extremal uncrowded hypergraphs.
\newblock {\em J. Combin. Theory Ser. A}, 32:321--335, 1982.

\bibitem{ram-ajtai80}
M.~Ajtai, J.~Koml{\'o}s, and E.~Szemer{\'e}di.
\newblock A note on {R}amsey numbers.
\newblock {\em J. Combin. Theory Ser. A}, 29:354--360, 1980.

\bibitem{ram-ajtai-sidon}
M.~Ajtai, J.~Koml{\'o}s, and E.~Szemer{\'e}di.
\newblock A dense infinite {S}idon sequence.
\newblock {\em European J. Combin.}, 2:1--11, 1981.

\bibitem{rt-balogh11}
J.~Balogh and J.~Lenz.
\newblock {The Ramsey-Tur\'{a}n Numbers of Graphs and Hypergraphs}.
\newblock { to appear in Israel Journal of Mathematics}.

\bibitem{beg-bollobas89}
B.~Bollob{\'a}s.
\newblock An extension of the isoperimetric inequality on the sphere.
\newblock {\em Elem. Math.}, 44:121--124, 1989.

\bibitem{beg-bollobas76}
B.~Bollob{\'a}s and P.~Erd\H{o}s.
\newblock On a {R}amsey-{T}ur\'an type problem.
\newblock {\em J. Combinatorial Theory Ser. B}, 21:166--168, 1976.

\bibitem{rt-erdos94}
P.~Erd\H{o}s, A.~Hajnal, M.~Simonovits, V.~T. S{\'o}s, and E.~Szemer{\'e}di.
\newblock Tur\'an-{R}amsey theorems and {$K_p$}-independence numbers.
\newblock {\em Combin. Probab. Comput.}, 3:297--325, 1994.

\bibitem{rt-erdos83}
P.~Erd\H{o}s, A.~Hajnal, V.~T. S{\'o}s, and E.~Szemer{\'e}di.
\newblock More results on {R}amsey-{T}ur\'an type problems.
\newblock {\em Combinatorica}, 3:69--81, 1983.

\bibitem{rt-erdos62}
P.~Erd\H{o}s and A.~Rogers.
\newblock The construction of certain graphs.
\newblock {\em Canad. J. Math.}, 14:702--707, 1962.

\bibitem{rt-erdos70}
P.~Erd\H{o}s and V.~T. S{\'o}s.
\newblock Some remarks on {R}amsey's and {T}ur\'an's theorem.
\newblock In {\em Combinatorial theory and its applications, {II} ({P}roc.
  {C}olloq., {B}alatonf\"ured, 1969)}, pages 395--404. North-Holland,
  Amsterdam, 1970.

\bibitem{rrl-erdos46}
P.~Erd\H{o}s and A.~H. Stone.
\newblock On the structure of linear graphs.
\newblock {\em Bull. Amer. Math. Soc.}, 52:1087--1091, 1946.

\bibitem{ram-erdos41}
P.~Erd\H{o}s and P.~Tur{\'a}n.
\newblock On a problem of {S}idon in additive number theory, and on some
  related problems.
\newblock {\em J. London Math. Soc.}, 16:212--215, 1941.

\bibitem{ram-fox10}
J.~Fox.
\newblock Complete minors and independence number.
\newblock {\em SIAM J. Discrete Math.}, 24:1313--1321, 2010.

\bibitem{ram-komlos82}
J.~Koml{\'o}s, J.~Pintz, and E.~Szemer{\'e}di.
\newblock A lower bound for {H}eilbronn's problem.
\newblock {\em J. London Math. Soc. (2)}, 25:13--24, 1982.

\bibitem{rt-simonovits01}
M.~Simonovits and V.~T. S{\'o}s.
\newblock Ramsey-{T}ur\'an theory.
\newblock {\em Discrete Math.}, 229:293--340, 2001.

\bibitem{rt-sudakov03}
B.~Sudakov.
\newblock A few remarks on {R}amsey-{T}ur\'an-type problems.
\newblock {\em J. Combin. Theory Ser. B}, 88:99--106, 2003.

\bibitem{rt-szemeredi72}
E.~Szemer{\'e}di.
\newblock On graphs containing no complete subgraph with {$4$} vertices.
\newblock {\em Mat. Lapok}, 23:113--116 (1973), 1972.

\bibitem{tur-turan41}
P.~Tur{\'a}n.
\newblock Eine {E}xtremalaufgabe aus der {G}raphentheorie.
\newblock {\em Mat. Fiz. Lapok}, 48:436--452, 1941.

\end{thebibliography}

\appendix
\section{Upper bounds}

 Erd\H{o}s, Hajnal, Simonovits, S\'{o}s, and Szemer\'{e}di~\cite{rt-erdos94} proved that $\theta_r(K_{r+s})\le \frac{s-1}{4r}$ for $s \leq 5$.  They also
proved two upper bounds: $\theta_3(K_8) \leq \frac{3}{11}$ and $\theta_3(K_9) \leq \frac{3}{10}$.
Extending on their techniques, we sketch a proof that $\theta_{10}(K_{16}), \theta_{12}(K_{19}) \le \frac{1}{8}$.  Combined with Theorem~\ref{mainconst}, this proves Proposition~\ref{appendixprop}.
Because of the similarity of the proofs, we do not give all the details of this proof; we only sketch the places where the argument differs from Section 4 in \cite{rt-erdos94}.

Assume $G$ is a counterexample. Apply Szemer\'edi's Regularity Lemma to $G$  to obtain  a cluster graph $H$.
Consider $H$ as a weighted graph, where the weight of an edge is the density of the pair of corresponding clusters.
If any vertex has weighted degree less than $(1/4+\epsilon/4)n$, delete it from $H$.

The remainder of the proof is focused on finding a copy of $K_{r+s}$ in $G$.  Let $w : E(H) \rightarrow [0,1]$ be
the weight function on $H$.  The proof technique is to force a bad configuration in $H$.
A \emph{configuration} is a weighted $K_t$.  If $t$ and the weights are chosen properly, the existence of the
configuration in $H$ will imply $G$ contains $K_{r+s}$.
The proof then comes down to a series of claims showing that $H$ must contain at least one bad configuration.

As an example, consider the following argument:
Assume $H$ contains a copy of $K_s$ with an edge of weight at least $1/r + \epsilon$ (the weights on the other
edges can be anything.) Let $A$ and $B$ be the corresponding clusters of the partition of $G$
with density $1/r + \epsilon$.  By the low $K_r$-independence number of $G$, the class $B$ must contain a copy of $K_r$.  By the lower bound on the density between
$A$ and $B$, there are two vertices, say $x$ and $y$, in this $K_r$ of $B$ which have a common neighborhood in $A$ of size
at least $\epsilon n/2$.  Again, the sublinear $K_r$-independence number of $G$ implies that  we can find a copy of $K_r$ in this common neighborhood.
Thus we have found a copy of $K_{r+2}$ in $G[A \cup B]$.  The Embedding Lemma can be used to extend this $K_{r+2}$ by
adding $s-2$ more vertices since we originally found a $K_s$ in $H$.

We introduce some shorthand notation for configurations.  $(K_t;a)$ means a copy of $K_t$  in $H$ with one edge with weight at least $a + \epsilon$,
$(K_t;a,b)$ means a copy of $K_t$ with one edge with weight at least $a + \epsilon$ and another edge with weight at least $b + \epsilon$,
$(K_t;a,\ldots,a)$ means a copy of $K_t$ with all edges with weight at least $a + \epsilon$.  We use the symbol $\leadsto$ to
mean that the existence of one configuration implies the existence of another.  A bad configuration is a configuration
whose existence implies a copy of $K_{r+s}$ in $G$.
It was proved in  \cite{rt-erdos94} that  the following configurations are bad.

\begin{lemma}\label{embeddingg}\cite{rt-erdos94}
For any non-negative integer $a$, $(K_{s-a+1};\frac{a}{r})$ is a bad configuration.
\end{lemma}

A simple corollary of this is that every edge has weight at most $(s-2)/r+\epsilon$.  The following lemmas
are our main new tools, showing more configurations are bad.

\begin{lemma}\label{embeddingtriangle}
For any non-negative integer $a$ and any reals $b,c$ satisfying $b\ge c$ and $b+(a+1)c/r > s-1$, $(K_3;\frac{a}{r},\frac{b}{r},\frac{c}{r})$ is a bad configuration.
\end{lemma}

\begin{proof}[Proof sketch]
Assume $H$ contains $(K_3;\frac{a}{r},\frac{b}{r},\frac{c}{r})$ and let $X,Y,Z$ be the parts in the regularity partition with $d(X,Y) = \frac{a}{r}$, $d(X,Z) = \frac{b}{r}$,
and $d(Y,Z) = \frac{c}{r}$.
By the $K_r$-independence number, $Y$ contains a copy of $K_r$.  Because of the density between $X$ and $Y$, this copy of $K_r$ contains
at least $a+1$ vertices with linear common neighborhood in $X$.  This common neighborhood contains a copy of $K_r$, so
we have found a copy of $K_{r+a+1}$ in $G[X\cup Y]$.  By averaging, there are $s$ vertices of this $K_{r+a+1}$
with linear common neighborhood in $Z$.  This common neighborhood contains a copy of $K_r$, so we have found a copy of $K_{r+s}$.
\end{proof}

\begin{lemma}\label{triangleconfigs}
$(K_2;\frac{s-3}{r}) \leadsto (K_3;\frac{s-3}{r},\frac{s-3}{r},\frac{s-3}{r})$.
\end{lemma}

\begin{proof}[Proof sketch]
First $(K_2;\frac{s-3}{r}) \leadsto (K_3;\frac{s-3}{r})$ since every vertex has weight at least $(s-1)/2r$.
Assume we have a triangle $uvz$ with an edge $uv$ with weight at least $(s-3)/r+\epsilon$.

We now show that $H$ contains $(K_3;\frac{s-3}{r},\frac{s-3}{r})$.
By Lemma~\ref{embeddingg}, there is no $K_4$ containing the triangle $uvz$ so for any other $x$, one of the weights from $xu,xv,xz$ is at most $\epsilon$ (meaning the edge is missing from $H$).
One of the other two weights on $xu$, $xv$, and $xz$ must be less than $(s-3)/r+\epsilon$, otherwise we would have the configuration $(K_3;\frac{s-3}{r},\frac{s-3}{r})$.
All edges have weight at most $(s-2)/r$.  Adding weight, one of $u$, $v$, or $z$ will violate the minimum weight degree condition in $H$.
A similar argument shows $(K_3;\frac{s-3}{r},\frac{s-3}{r}) \leadsto (K_3;\frac{s-3}{r},\frac{s-3}{r},\frac{s-3}{r})$.
\end{proof}

\begin{lemma}\label{triangleisbad}
If $3(s-2) > r$, then $(K_3;\frac{s-3}{r},\frac{s-3}{r},\frac{s-3}{r})$ is a bad configuration.
\end{lemma}

\begin{proof}
Let $uvz$ be a triangle with all edges with weight at least $\frac{s-3}{r}$.  By Lemma~\ref{embeddingg}, there is no $K_4$ containing this triangle.
By the minimum weight condition, there is some vertex $x$ sending $\frac{3(s-1)}{2r}$ to $uvz$ and since $uvzx$ does not form a $K_4$,
$w(x,u) + w(x,v) \geq \frac{3(s-1)}{4r}$.  Now apply Lemma~\ref{embeddingtriangle} with $a=s-3$ and $b=c= \frac{3(s-1)}{4}$ to obtain a contradiction.
\begin{align*}
\frac{3(s-1)}{4} + \frac{3(s-2)(s-1)}{4r} > s-1.
\end{align*}
\end{proof}

The combination of the above two lemmas shows that if $3(s-2) > r$ (which is true for $r=10,s=6$ and $r=12,s=7$)
then $(K_2;\frac{s-3}{r})$ is a bad configuration.  In other words, every edge in $H$ has weight
at most $(s-3)/r+\epsilon$.

\begin{proposition}
$\theta_{10}(K_{16}) \leq \frac{1}{8}$.
\end{proposition}

\begin{proof}
Fix $r = 10$ and $s=6$.  First, Lemma~\ref{embeddingg} shows $H$ is $K_7$-free (with $a=0$.)  
Lemmas~\ref{triangleconfigs} and \ref{triangleisbad} show all edges have weight
at most $3/10 + \epsilon$.  Then Tur\'an's Theorem implies the total weight is at most 
\begin{align*}
\frac{1}{2} \left(1-\frac{1}{6}\right)\left(\frac{3}{10}+\epsilon\right)\leq \frac{1}{8}+\epsilon,
\end{align*}
a contradiction.
\end{proof}

\begin{proposition}
$\theta_{12}(K_{19}) \leq \frac{1}{8}$.
\end{proposition}

\begin{proof}[Proof sketch]
Fix $r=12$ and $s=7$.  Lemma~\ref{triangleconfigs} and Lemma~\ref{triangleisbad} show
that all edges have weight at most $4/12+\epsilon$. We claim without proof that several configurations are bad.

\begin{itemize}
\item $(K_3;\frac{3.75}{12},\dots,\frac{3.75}{12})$ is a bad configuration.
\item $(K_4;\frac{3}{12},\dots,\frac{3}{12})$ is a bad configuration.
\item $(K_5;\frac{2}{12},\dots,\frac{2}{12})$ is a bad configuration.
\item $H$ is $K_6$-free.
\end{itemize}

All of the above facts have very similar proofs.  As an example, Lemma~\ref{embeddingg} shows
$(K_5;\frac{3}{12})$ is a bad configuration.  So consider any $K_4$ with all edges with weight at least $3/12$.
There is some vertex $x$ with at least weight $1$ towards the $K_4$.  But $x$ is not adjacent to all of the $K_4$,
so the weight is distributed over three edges.  But the maximum weight on an edge is $4/12$, a contradiction.

So assume the four bullet points above are true and count up the total weight.  The edges with weight more than $3.75/12$
are triangle free, so their total weight is at most $\frac{n^2}{4} \frac{4}{12}$.  The edges with weight more than
$3/12$ are $K_4$-free so there are at most $n^2/3$ of them, and there are at most $n^2/3 - n^2/4$ edges
that are uncounted by the bound $\frac{n^2}{4} \frac{4}{12}$ and these have weight at most $3.75/12$.  Thus the total weight is at most
\begin{align*}
\frac{n^2}{4} \frac{4}{12} + \left( \frac{n^2}{3} - \frac{n^2}{4} \right) \frac{3.75}{12} + \left( \frac{3n^2}{8} - \frac{n^2}{3} \right) \frac{3}{12} +
\left(\frac{2n^2}{5} - \frac{3n^2}{8} \right) \frac{2}{12} < \frac{n^2}{8}.
\end{align*}
\end{proof}

\end{document}